\documentclass[letterpaper,10pt]{article}
\usepackage[textwidth=11.5cm, textheight=22cm, top=2.0cm, centering]{geometry}
\usepackage{setspace}
\usepackage[T1]{fontenc}
\usepackage{lmodern}
\usepackage[english]{babel}
\usepackage{microtype} 
\usepackage[scaled=0.95]{inconsolata} 
\makeatletter
\renewcommand{\texttt}[1]{{\ttfamily #1}}

\renewcommand{\mathtt}[1]{\text{\texttt{#1}}} 
\makeatother
\usepackage{csquotes}
\usepackage{etoolbox}
\usepackage{bbding}
\usepackage{changepage}
\usepackage{booktabs}
\usepackage{graphicx}
\usepackage[x11names, dvipsnames]{xcolor}
\definecolor{Linkz}{RGB}{30, 110, 170}
\definecolor{Darkenta}{RGB}{185, 35, 90}
\definecolor{Magentz}{RGB}{255, 35, 170}
\definecolor{Lightenta}{RGB}{254, 232, 255}
\definecolor{Reference}{RGB}{35, 180, 90}
\definecolor{Periwinkle}{RGB}{102, 51, 255}
\definecolor{yello}{RGB}{255, 245, 230}
\definecolor{Greeno}{RGB}{0, 140, 100}
\definecolor{Leeno}{RGB}{239, 255, 232}
\definecolor{Nicegreen}{RGB}{100, 200, 130}
\usepackage{amsmath, amsthm, amscd, amssymb, stmaryrd}
\usepackage{stackengine}
\usepackage{bussproofs}
\usepackage[square]{natbib}

\usepackage[
colorlinks=true,
linkcolor=Darkenta,
citecolor=Greeno,
urlcolor=Darkenta,
]{hyperref}
\usepackage{titlesec} 
\usepackage{abstract} 
\usepackage{enumitem} 
\usepackage{ccicons} 
\usepackage{float}
\usepackage{placeins} 
\usepackage{tikz} 
\usepackage{tikz-cd}
\usetikzlibrary{positioning,fit,patterns}
\usetikzlibrary{arrows.meta} 
\usepackage{adjustbox} 
\usepackage{pgfmath} 
\usepackage{ifthen} 

\newtheoremstyle{upright}
{6pt plus 2pt minus 2pt} 
{6pt plus 2pt minus 2pt} 
{\normalfont} 
{} 
{\bfseries} 
{.} 
{.5em} 
{} 
\theoremstyle{upright}
\theoremstyle{upright}
\newtheorem{theorem}{Theorem}[subsection]
\newtheorem{remark}[theorem]{Remark}

\newtheorem{definition}[theorem]{Definition}

\newtheorem{proposition}[theorem]{Proposition}
\newtheorem{lemma}[theorem]{Lemma}

\newtheorem{corollary}[theorem]{Corollary}

\newtheorem{mechanization}[theorem]{Mechanization}

\makeatletter
\renewenvironment{proof}[1][Proof]{%
	\par\pushQED{\qed}%
	\normalfont
	\topsep6\p@\@plus6\p@\relax
	\trivlist
	\item[\hskip\labelsep\slshape #1\@addpunct{.}]%
}{%
	\popQED\endtrivlist\@endpefalse
}

\makeatother

\usepackage{algpseudocode}
\usepackage[most]{tcolorbox}
\usepackage{listings}
\newtcolorbox{breakbox}[2][]{%
	breakable,
	={#2},
	fonttitle=\bfseries,
	colback=white,
	colframe=black!20,
	coltitle=black,
	colbacktitle=white,
	boxrule=0.5pt,
	arc=0pt,
	boxsep=7pt,
	left=3pt,
	right=2pt,
	top=2pt,
	bottom=4pt,
	fontupper=\small\sffamily, 
	#1
}

\AtBeginEnvironment{quotation}{\sffamily}
\renewenvironment{quotation}
{\small\vspace{0.5em}\begin{adjustwidth}{4em}{4em}%

		\setlength{\parindent}{0pt}%
		\setlength{\parskip}{\medskipamount}%
	}
	{\end{adjustwidth}\vspace{1em}}
\newcommand{\customsectionstyle}[2]{%
	\titleformat{\section}[block]
	{\normalfont\fontsize{#1}{1.2\dimexpr#1\relax}\selectfont\centering}
	{\thesection}{1em}%
	{%
		\ifthenelse{\equal{#2}{true}}{\MakeUppercase}{\relax}%
	}%
}
\newcommand{\customsectionspacing}[3]{%
	\titlespacing*{\section}{#1}{#2}{#3}%
}
\newcommand{\customsubsectionstyle}[2]{%
	\titleformat{\subsection}[block]
	{\normalfont\fontsize{#1}{1.2\dimexpr#1\relax}\selectfont\centering}
	{\thesubsection}{1em}%
	{%
		\ifthenelse{\equal{#2}{true}}{\MakeUppercase}{\relax}%
	}%
}
\newcommand{\customsubsectionspacing}[3]{%
	\titlespacing*{\subsection}{#1}{#2}{#3}%
}
\customsectionstyle{14pt}{true}
\customsectionspacing{0pt}{3.5ex plus 1ex minus 0.2ex}{2.5ex}
\customsubsectionstyle{11pt}{true}
\customsubsectionspacing{0pt}{4ex plus 1ex minus 0.2ex}{2ex}
\usepackage{caption}

\makeatletter
\newcommand{\shorttitle}[1]{\def\@shorttitle{#1}}
\newcommand{\email}[1]{\def\@email{#1}}
\newcommand{\metadata}[1]{\def\@metadata{#1}}
\renewcommand{\maketitle}{%
	\begin{center}
		\vfill
		{\fontsize{18pt}{19pt}\selectfont \@title \par}
		\vspace{1em}
		{\normalsize \@author \par}
		\vspace{0.1em}
		{\normalsize \@date \par}
	\end{center}
}
\makeatother

\begin{document}
\newcommand{\T}{\operatorname{T}}

\newcommand{\Formula}{\ensuremath{\mathsf{Form}}}
\newcommand{\Bott}{\ensuremath{\bot}}
\newcommand{\Imp}{\ensuremath{\to}}
\newcommand{\Neg}[1]{\ensuremath{\lnot #1}}
\newcommand{\Eqv}[3][C]{\ensuremath{#2 \simeq_{#1} #3}}
\newcommand{\LEM}[1][C]{\ensuremath{\mathsf{LEM}(#1)}}
\newcommand{\LocalLEM}[2][C]{\ensuremath{\mathsf{LocalLEM}_{#1}(#2)}}
\newcommand{\MP}[1][C]{\ensuremath{\mathsf{MP}(#1)}}
\newcommand{\Det}[3][C]{\ensuremath{\mathsf{Det}_{#1}(#2,#3)}}
\newcommand{\Cons}[1][C]{\ensuremath{\mathsf{Cons}(#1)}}
\newcommand{\SClass}[1][C]{\ensuremath{\mathsf{Signed}(#1)}}
\newcommand{\MemDec}[1][C]{\ensuremath{\mathsf{MemDec}(#1)}}
\newcommand{\RefComp}[1][C]{\ensuremath{\mathsf{RefComp}(#1)}}
\newcommand{\RefInd}[1][C]{\ensuremath{\mathsf{RefInd}(#1)}}
\newcommand{\NegFP}[2][C]{\ensuremath{\mathsf{NegFP}_{#1}(#2)}}
\newcommand{\evapp}[2]{\ensuremath{\mathsf{eval}(#1,#2)}}
\newcommand{\Defeq}{\ensuremath{\mathrel{\triangleq}}}


\title{\uppercase{Remarks on\\ Primitive Regulation}}

\author{Milan Rosko}
\date{July 2026}

\maketitle

\begin{center}\footnotesize{
		ORCID: \href{https://orcid.org/0009-0003-1363-7158}{\footnotesize\textsf{0009-0003-1363-7158}}\\
}
\end{center}

\begin{abstract} \vspace{-1ex}\footnotesize{
	We prove, and mechanize in \textsc{Rocq}, an obstruction to closure-level \textsc{Excluded Middle} for primitive regulators $C : \mathsf{Form} \to \mathsf{Prop}$ over the closed implication--falsity fragment $A,B ::= \bot \mid A \to B$. Write $\mathsf{LEM}(C)$ for the  that $C(A)\lor C(\lnot A)$ hold for every formula $A$. If $C$ is closed under \textsc{Modus Ponens}, is consistent, and admits a formula $B$ satisfying $B\simeq_C\lnot B$, where $A \simeq_C B$ abbreviates $C(A \to B) \land C(B \to A)$, then $\mathsf{LEM}(C)$ is impossible. In fact, consistency excludes both $C(B)$ and $C(\lnot B)$, so the global conclusion uses only the instance of closure-level \textsc{Excluded Middle} at $B$.
	}
\end{abstract}
\section{Introduction}
\label{sec:introduction}

\subsection{Intuition}
\label{sec:introduction:intuition}

Consider how this question already \emph{pressures} the form of answer it solicits:
\begin{quotation}
	Every formula is true or false. Is that claim itself true or false?
\end{quotation}
A \emph{yes} would endorse unrestricted closure-level \textsc{Excluded Middle}, while a \emph{no} would subject that totality principle to the same true-or-false demand.

\subsection{Background}
\label{sec:introduction:background}

Impossibility recurs across several formal settings. \textsc{Russell's Antinomy} \citep{russell08,russell27} exposes unrestricted comprehension; \textsc{Incompleteness} \citep{godel31,rosser36} turns arithmetized representation of syntax back upon derivability; \textsc{Undefinability} \citep{tarski1933} turns \textsc{Diagonalization} against a proposed truth predicate; computability-theoretic limitations \citep{church36,turing37,kleene52,rice53} combine effective indexing with self-application; and \textsc{Löb's Theorem} \citep{loeb55} controls self-referential provability through derivability conditions. Their objects and conclusions differ, but each motivates the study of fixed-points and total classification.

The present paper isolates this \emph{pressure} at the level of primitive regulation. Its question is whether a formula-indexed acceptance predicate can simultaneously admit a closure-equivalence negation fixed-point, support detachment, remain consistent, and satisfy closure-level \textsc{Excluded Middle}.

Prior to the formal definitions, let $C$ denote a primitive acceptance predicate
\begin{equation}
	\mathrm{Regulator}\;C,
\end{equation}
on the closed implication--falsity fragment. We use \emph{regulator} in a sparse sense, as a \textsc{Black Box} \citep{ashby56}: a constrained interface that admits or rejects proposed responses. In the checked realization developed below, finite proof scripts serve as instructions, formulas as proposed outputs, and a Boolean checker determines acceptance. The derived release gate returns an accepted output or nothing; any listener that interprets a released formula as behavior remains external to the formal substrate.

Thus $C(A)$ means that the formula $A$ is accepted. The following demands can be made. (i) A formula $B$ satisfies $B\simeq_C\Neg B$. (ii) $C$ supports \emph{detachment}: if $C(A \Imp D)$ and $C(A)$, then $C(D)$. (iii) $C$ is \emph{consistent}: it does not accept $\Bott$. (iv) $C$ satisfies closure-level \textsc{Excluded Middle}: for every formula $A$, either $C(A)$ or $C(\Neg A)$. Assume now a formula $B$ satisfying
\begin{equation}
	B \simeq_C \Neg B,
\end{equation}
so that $C$ accepts both $B\Imp\Neg B$ and $\Neg B\Imp B$. Closure-level \textsc{Excluded Middle} then supplies one of the two acceptance facts
\begin{equation}
	C(B) \quad\text{or}\quad C(\Neg B).
\end{equation}
Whichever branch holds, detachment along the accepted fixed-point implications produces both $C(B)$ and $C(\Neg B)$, and hence $C(\Bott)$. Consistency turns this collapse into a contradiction.

The global principle is therefore obstructed at a single formula. The result is stronger than the bare failure of a universal schema: once the fixed-point is present, one instance of closure-level \textsc{Excluded Middle} already carries the entire contradiction.

\subsection{Contributions}
\label{sec:introduction:contributions}

In the companion \textsc{Rocq} development, \texttt{M001} and \texttt{L001} are module identifiers. \texttt{M001} supplies the executable syntax, proof-checking substrate, and generic accepted-output gate. \texttt{L001} supplies the abstract closure-obstruction layer over the primitive syntax.

We make four contributions. (i) We derive the failure of closure-level \textsc{Excluded Middle} for every consistent, modus-ponens-closed primitive regulator admitting a closure-equivalence negation fixed-point, and show that the proof uses only the instance of closure-level \textsc{Excluded Middle} at that point, the fixed-point direction aligned with the selected branch, and the corresponding detachment instances. (ii) We show that signed classification and refutation completeness are obstructed at the fixed-point, while ordinary membership decision alone is not implicated and the weaker notion of a refutation-sound Boolean indicator is trivially inhabited. (iii) We prove the four grouped hypotheses irredundant by explicit elementary acceptance predicates. (iv) We distinguish the abstract boundary from the stronger \textsc{K}/\textsc{S} checked-derivability consequence supplied by \texttt{M001}, and verify the proof-critical contracts in \textsc{Rocq}.

\subsection{Roadmap}
\label{sec:introduction:roadmap}

The exposition follows this division. Section~\ref{sec:m001} develops the \texttt{M001} substrate by fixing the closed $\Bott/\Imp$ syntax, the closure vocabulary, and the checked-derivability instance. Section~\ref{sec:l001} develops the \texttt{L001} obstruction layer by proving the local collapse, deriving the obstruction to closure-level \textsc{Excluded Middle}, separating the stronger checked consequence, repackaging the fixed-point through a supplied goal frame, developing the classification consequences, and proving irredundancy of the four grouped hypotheses. Section~\ref{sec:mech} records the \texttt{M001} substrate and the six certified \texttt{L001} contracts. Section~\ref{sec:discussion} explains why the global abstract obstruction is already determined by one local branch.

\section{Substrate}
\label{sec:m001}

\subsection{Object syntax}
\label{sec:m001:syntax}

\begin{definition}[Syntax]
\label{def:syntax}
	Working in a constructive metatheory \citep{troelstra88}, such as the \textsc{Calculus of Constructions} \citep{coquand88}, we take as our object language a minimal propositional language containing only falsity and implication.
	\begin{equation}
		\label{eq:syntax}
		A,B \;::=\; \Bott \quad\mid\quad A \Imp B, \quad \Neg A \;\Defeq\; A \Imp \Bott.
	\end{equation}
	We write $\Formula$ for the type of formulas; formulas are finite trees over $\{\Bott, \Imp\}$. There are no atoms, propositional variables, quotation constructors, object-level substitution operations, or semantic truth predicates in this language.
\end{definition}

\begin{lemma}[No formula is its own negation]
\label{lem:no-self-neg}
	For every formula,
	\begin{equation}
	\forall A : \Formula, \; A \neq (A \Imp \Bott).
	\end{equation}
\end{lemma}

\begin{proof}
	Let $\lvert A\rvert$ be the number of constructors in the finite formula tree $A$. Then
	\begin{equation}
		\lvert A\Imp\Bott\rvert
		= 1+\lvert A\rvert+\lvert\Bott\rvert
		> \lvert A\rvert,
	\end{equation}
	so $A$ and $A\Imp\Bott$ cannot be equal.
\end{proof}

\begin{remark}
	Lemma~\ref{lem:no-self-neg} forces every fixed-point claim below to be stated via $\simeq_C$, never as a syntactic identity of formulas.
\end{remark}

\subsection{Closure predicates}
\label{sec:m001:closure}

\begin{definition}[Primitive regulator with closure predicate]
\label{def:primitive-regulator}
	A \emph{primitive regulator} on the formula language is a predicate
	\begin{equation}
		C \;:\; \Formula \to \mathsf{Prop}.
		\label{eq:primitive-regulator}
	\end{equation}
	We write $C(A)$, read “$C$ accepts $A$.” A primitive regulator is a proposition-valued acceptance predicate, also called a \emph{closure predicate} when emphasizing its closure properties. All closure rules used below are explicit hypotheses, and every result is parametric in $C$.
\end{definition}

\begin{remark}
	Throughout, omitted parameters in $\LEM{}$, $\MP{}$, $\Det{A}{D}$, $\Cons{}$, $\LocalLEM{B}$, $\SClass{}$, $\MemDec{}$, $\RefComp{}$, $\RefInd{}$, and $\NegFP{B}$ refer to the ambient closure predicate $C$.
\end{remark}

\begin{definition}[Closure equivalence]
\label{def:equiv}
	For formulas $A$ and $B$, closure equivalence means mutual acceptance of the corresponding implications. We define
\begin{equation}
	\label{eq:closure-equiv}
	\Eqv{A}{B} \;\Defeq\; C(A \Imp B) \;\land\; C(B \Imp A).
\end{equation}
	Thus $A$ and $B$ are equivalent relative to $C$ when $C$ accepts both directions between them.
\end{definition}

\begin{remark}
	$\simeq_C$ records exactly the two accepted implications between $A$ and $B$. The name does not endow a bare $C$ with reflexivity, transitivity, congruence, or substitution. By Lemma~\ref{lem:no-self-neg}, $A = \Neg A$ never holds, whereas $\Eqv{A}{\Neg A}$ may, depending on $C$.
\end{remark}

\begin{definition}[Closure-equivalence negation fixed-point]
\label{def:negfp}
	A formula $B : \Formula$ is a \emph{closure-equivalence negation fixed-point} for $C$ when
	\begin{equation}
	\label{eq:negfp-rel}
		\NegFP{B} \;\Defeq\; \Eqv{B}{\Neg B}.
	\end{equation}
	Equivalently,
	\begin{equation}
		\NegFP{B} \;\Defeq\; C(B\Imp\Neg B)\land C(\Neg B\Imp B).
	\end{equation}
	This is the fixed-point hypothesis used by the obstruction theorem.
\end{definition}

\begin{definition}[\textsc{Modus Ponens} and consistency]
\label{def:mp-cons-lem}
	The structural properties of $C$ are
\begin{equation}
	\begin{array}{lll}
		\MP&\Defeq& \forall A,B : \Formula,\; \bigl(C(A \Imp B) \land C(A)\bigr) \to C(B),\\[1ex]
		\Cons &\Defeq& C(\Bott) \to \bot.
	\end{array}
\end{equation}
\end{definition}

\begin{definition}[Local detachment]
\label{def:local-detachment}
	For formulas $A,D$, the corresponding detachment instance is
	\begin{equation}
		\Det{A}{D}
		\;\Defeq\;
		C(A\Imp D)\to C(A)\to C(D).
	\end{equation}
	Thus \MP{} entails \Det{A}{D} for every $A,D$, while a local argument may name only the instances it uses.
\end{definition}

\begin{definition}[Closure-level \textsc{Excluded Middle}]
\label{def:closure-lem}
	Closure-level \textsc{Excluded Middle} for $C$ is the totality principle
	\begin{equation}
		\label{eq:closure-lem}
		\LEM{} \;\Defeq\; \forall A : \Formula,\; C(A)\lor C(\Neg A).
	\end{equation}
	It requires the regulator to accept a positive or negative side for every object formula. The principle is a property of $C$; the obstruction to it below is proved in the constructive metatheory.
\end{definition}

\begin{definition}[Local closure-level \textsc{Excluded Middle}]
\label{def:local-lem}
	For a fixed formula $B$, local closure-level \textsc{Excluded Middle} at $B$ is the single instance
	\begin{equation}
		\LocalLEM{B} \;\Defeq\; C(B)\lor C(\Neg B).
	\end{equation}
	Clearly, \LEM{} entails \LocalLEM{B} for every $B$. The local obstruction below assumes only this one instance.
\end{definition}

\begin{definition}[Signed classification, membership decision, and refutation-sound indicators]
\label{def:classification}
	A \emph{signed classification} for $C$ is a Boolean classifier whose two verdicts carry internal closure certificates:
	\begin{equation}
		\label{eq:signed-classification}
		\begin{aligned}
			\SClass \;\Defeq\;{}& \exists s : \Formula \to \{\mathsf{tt},\mathsf{ff}\},\; \forall A : \Formula,\\
			&\bigl(s(A)=\mathsf{tt}\to C(A)\bigr)
			\land \bigl(s(A)=\mathsf{ff}\to C(\Neg A)\bigr).
		\end{aligned}
	\end{equation}
	An \emph{ordinary membership decision} instead decides the proposition $C(A)$:
	\begin{equation}
		\label{eq:membership-decision}
		\begin{aligned}
			\MemDec \;\Defeq\;{}& \exists m : \Formula \to \{\mathsf{tt},\mathsf{ff}\},\; \forall A : \Formula,\\
			&m(A)=\mathsf{tt}\;\longleftrightarrow\;C(A).
		\end{aligned}
	\end{equation}
	The additional bridge from non-membership to an accepted object-language negation is
	\begin{equation}
		\label{eq:refutation-completeness}
		\RefComp \;\Defeq\; \forall A : \Formula,\; \neg C(A) \to C(\Neg A).
	\end{equation}
	A \emph{refutation-sound Boolean indicator} certifies an accepted object-language negation whenever it returns true, but has no availability, completeness, or nontriviality requirement:
	\begin{equation}
		\begin{aligned}
			\RefInd \;\Defeq\;{}& \exists r : \Formula \to \{\mathsf{tt},\mathsf{ff}\},\\
			& \forall A : \Formula,\; r(A) = \mathsf{tt} \to C(\Neg A).
		\end{aligned}
	\end{equation}
\end{definition}

\begin{remark}
	Signed classification and membership decision are intentionally distinct. A false signed verdict already certifies $C(\Neg A)$; a false membership verdict supplies only $\neg C(A)$. The latter reaches an accepted negation only through \RefComp{}. A refutation-sound indicator is weaker still: it need never return true. Since no model of computation is fixed, none of these abstract conditions alone states recursion-theoretic undecidability.
\end{remark}

\begin{definition}[\textsc{Regulator Theory}]
\label{def:regulator-theory}
	The mechanized substrate also provides a concrete syntactic instance. A \textsc{Regulator Theory} is a pair
	\begin{equation}
		R \;\Defeq\; (\pi,\T),
	\end{equation}
	where $\pi$ selects either the minimal \textsc{K}/\textsc{S} profile or the \textsc{K}/\textsc{S} profile extended with the explicit \textsc{EFQ} schema, and $\T$ is a Boolean axiom-membership function over object formulas. The interface need not carry a finite enumeration; the mechanization provides \texttt{FiniteAxiomSet} separately when the axiom source itself must be finite data. A context $\Gamma$ supplies local assumptions but is not part of the theory. Together, $R$ and $\Gamma$ induce the closure predicate
	\begin{equation}
		C_{R,\Gamma}(A) \;\Defeq\; \Gamma \vdash_R A,
	\end{equation}
	where $\vdash_R$ denotes checked derivability in $R$. This gives one family of closure predicates to which the abstract vocabulary may be applied. The obstruction theorem itself remains parametric in $C$ and does not inspect the checker or assume that every $C$ arises from a \textsc{Regulator Theory}.
\end{definition}

\begin{proposition}[Checked derivability is closed under detachment]
\label{prop:m001-checked-mp}
	For every \textsc{Regulator Theory} $R$ and context $\Gamma$, the induced closure predicate satisfies
	\begin{equation}
		\MP[C_{R,\Gamma}].
	\end{equation}
\end{proposition}

\begin{proof}
	Suppose $C_{R,\Gamma}(A\Imp D)$ and $C_{R,\Gamma}(A)$. These are checked derivations of $A\Imp D$ and $A$ in the same theory and context. Composing their proof scripts with the checker’s \textsc{Modus Ponens} constructor yields a checked derivation of $D$, hence $C_{R,\Gamma}(D)$.
\end{proof}

\begin{proposition}[Checked negation fixed-points collapse]
\label{prop:m001-checked-negfp-collapse}
	For every \textsc{Regulator Theory} $R$, context $\Gamma$, and formula $B$,
	\begin{equation}
		\NegFP[C_{R,\Gamma}]{B} \;\longrightarrow\; C_{R,\Gamma}(\Bott).
	\end{equation}
\end{proposition}

\begin{proof}
	Write $P:=B\Imp(B\Imp\Bott)$. Every \textsc{K}/\textsc{S} profile derives the identity $B\Imp B$ and the \textsc{S}-instance
	\begin{equation}
		P\Imp\bigl((B\Imp B)\Imp(B\Imp\Bott)\bigr).
	\end{equation}
	The first half of $\NegFP[C_{R,\Gamma}]{B}$ is $C_{R,\Gamma}(P)$. Two checked detachments therefore give $C_{R,\Gamma}(B\Imp\Bott)$, that is, $C_{R,\Gamma}(\Neg B)$. The reverse fixed-point implication $C_{R,\Gamma}(\Neg B\Imp B)$ then gives $C_{R,\Gamma}(B)$, and one final detachment yields $C_{R,\Gamma}(\Bott)$.
\end{proof}

\section{Obstruction}
\label{sec:l001}

\subsection{Local collapse}
\label{sec:l001:collapse}

\begin{theorem}[Branchwise goal-relative collapse]
\label{thm:branchwise-goal-relative-collapse}
	Let $B,G : \Formula$. Assume
	\begin{equation}
		\Det{B}{B\Imp G}, \quad \Det{B\Imp G}{B}, \quad \Det{B}{G},
	\end{equation}
	and
	\begin{equation}
		\bigl(C(B)\land C(B\Imp(B\Imp G))\bigr)
		\;\lor\;
		\bigl(C(B\Imp G)\land C((B\Imp G)\Imp B)\bigr).
	\end{equation}
	Then
	\begin{equation}
		C(G).
	\end{equation}
\end{theorem}

\begin{proof}
	In the first branch, \Det{B}{B\Imp G} gives $C(B\Imp G)$, and \Det{B}{G} then gives $C(G)$. In the second branch, \Det{B\Imp G}{B} gives $C(B)$, after which \Det{B}{G} again gives $C(G)$. Thus each branch uses only its corresponding fixed-point implication, one branch-specific detachment instance, and the final instance \Det{B}{G}.
\end{proof}

\begin{theorem}[Goal-relative branch collapse]
\label{thm:goal-relative-branch-collapse}
	Let $B,G : \Formula$. Assume \MP{},
	\begin{equation}
		\Eqv{B}{B\Imp G},
	\end{equation}
	and the local branch
	\begin{equation}
		C(B)\lor C(B\Imp G).
	\end{equation}
	Then
	\begin{equation}
		C(G).
	\end{equation}
\end{theorem}

\begin{proof}
	The hypothesis \MP{} supplies the three local detachment instances in Theorem~\ref{thm:branchwise-goal-relative-collapse}, while the closure equivalence supplies both oriented implications. If $C(B)$ holds, use the first branch of that theorem; if $C(B\Imp G)$ holds, use the second.
\end{proof}

\begin{theorem}[Local branch collapse]
\label{thm:local-branch-collapse}
	Assume \MP{}, $\NegFP{B}$, and \LocalLEM{B}. Then
	\begin{equation}
		C(\Bott).
	\end{equation}
\end{theorem}

\begin{proof}
	Apply Theorem~\ref{thm:goal-relative-branch-collapse} with $G:=\Bott$. Since $\Neg B$ abbreviates $B\Imp\Bott$, its closure-equivalence and local-branch hypotheses are exactly $\NegFP{B}$ and \LocalLEM{B}.
\end{proof}

\begin{theorem}[Fixed-point gap]
\label{thm:fixedpoint-gap}
	Assume \MP{}, \Cons{}, and $\NegFP{B}$. Then
	\begin{equation}
		\neg C(B)\;\land\;\neg C(\Neg B).
	\end{equation}
\end{theorem}

\begin{proof}
	If $C(B)$ held, it would supply the left side of \LocalLEM{B}; Theorem~\ref{thm:local-branch-collapse} would then give $C(\Bott)$, contrary to \Cons{}. The same argument from $C(\Neg B)$ uses the right side of \LocalLEM{B}. Hence neither branch is accepted.
\end{proof}

\begin{theorem}[\textsc{Aporetic Lemma}]
\label{thm:core-diagonal-obstruction}
	A closure-equivalence negation fixed-point, detachment, consistency, and local bivalence at the fixed-point are incompatible,
	\begin{equation}
	\MP{} \land \Cons{} \land \NegFP{B} \land \LocalLEM{B} \implies \bot.
	\end{equation}
\end{theorem}

\begin{proof}
	Theorem~\ref{thm:fixedpoint-gap} gives $\neg C(B)$ and $\neg C(\Neg B)$. Either side of \LocalLEM{B} therefore yields $\bot$.
\end{proof}

\begin{corollary}[\textsc{Excluded Middle} fails at a negation fixed-point]
\label{cor:fixedpoint-lem-obstruction}
	Assume \MP{}, \Cons{}, and $\NegFP{B}$. Then
	\begin{equation}
		\neg\LEM{}.
	\end{equation}
\end{corollary}

\begin{proof}
	Suppose \LEM{} holds. Specializing it at $B$ gives \LocalLEM{B}. Theorem~\ref{thm:core-diagonal-obstruction} then yields $\bot$.
\end{proof}

\begin{corollary}[Consistent checked regulators have no negation fixed-point]
\label{cor:checked-regulator-obstruction}
	For every \textsc{Regulator Theory} $R$ and context $\Gamma$, if $C_{R,\Gamma}$ is consistent, then
	\begin{equation}
		\neg\exists B : \Formula,\; \NegFP[C_{R,\Gamma}]{B}.
	\end{equation}
\end{corollary}

\begin{proof}
	Any such $B$ would give $C_{R,\Gamma}(\Bott)$ by Proposition~\ref{prop:m001-checked-negfp-collapse}, contrary to consistency.
\end{proof}

Thus the abstract four-way obstruction is nontrivial at the level of arbitrary closure predicates, whereas the \textsc{K}/\textsc{S} checked instance has a stricter boundary: its fixed-point premise already conflicts with consistency. Neither \texttt{M001} nor \texttt{L001} constructs an instance of $\mathsf{Frame}_{\Bott}(C_{R,\Gamma})$; by the corollary, no consistent checked instance can do so. Checked derivability carries structural principles not assumed by the abstract \texttt{L001} theorem.

\subsection{Repackaging a fixed-point through a supplied goal frame}
\label{sec:l001:frame-extraction}

\begin{definition}[Supplied goal frame]
\label{def:eval-goal-frame}
	Fix an arbitrary type $\mathsf{Code}$ and an operation
	\begin{equation}
		\mathsf{eval} : \mathsf{Code} \to \mathsf{Code} \to \Formula.
	\end{equation}
	For a fixed goal formula $G$, the frame asserts the behavior
	\begin{equation}
		x \mapsto \evapp{x}{x} \Imp G.
	\end{equation}
	We write the frame as
	\begin{equation}
		\label{eq:eval-goal-frame}
		\mathsf{Frame}_G(C) \;\Defeq\; \exists c\; \forall x\; \Eqv{\evapp{c}{x}}{\evapp{x}{x} \Imp G}.
	\end{equation}
	The obstruction uses the case $G=\Bott$.
\end{definition}

\begin{remark}
	Conversely, given any $B$ with $\Eqv{B}{B\Imp G}$, take a singleton code type $\{\star\}$ and the constant operation
	\begin{equation}
	\evapp{\star}{\star}:=B,
	\end{equation}
	Then the resulting goal frame holds. Thus, with freely chosen code and evaluation data, goal frames and goal-relative fixed-points are interderivable.
\end{remark}

\begin{remark}
	Closure equivalence is essential. Replacing $\simeq_C$ by formula equality would produce $B=\Neg B$, contradicting Lemma~\ref{lem:no-self-neg}.
\end{remark}

\begin{theorem}[Self-specialization of a bottom goal frame]
\label{thm:eval-bottom-negfixp}
	Under $\mathsf{Frame}_{\Bott}(C)$,
	\begin{equation}
		\exists B : \Formula,\; \NegFP{B}.
	\end{equation}
\end{theorem}

\begin{proof}
	Unpack $\mathsf{Frame}_{\Bott}(C)$ to obtain a code $c$ such that
	\begin{equation}
		\forall x,\; \Eqv{\evapp{c}{x}}{\evapp{x}{x}\Imp\Bott}.
	\end{equation}
	Set $B \Defeq \evapp{c}{c}$. Specializing the frame property at $x:=c$ gives
	\begin{equation}
		\Eqv{B}{B\Imp\Bott},
	\end{equation}
	which is exactly $\NegFP{B}$ because $\Neg B \Defeq B\Imp\Bott$.
\end{proof}

\begin{corollary}[Conditional obstruction from a bottom goal frame]
\label{cor:eval-bottom-lem-obstruction}
	Assume $\mathsf{Frame}_{\Bott}(C)$, \MP{}, and \Cons{}. Then
	\begin{equation}
		\neg\LEM{}.
	\end{equation}
\end{corollary}

\begin{proof}
	Theorem~\ref{thm:eval-bottom-negfixp} extracts a formula $B$ with $\NegFP{B}$. Corollary~\ref{cor:fixedpoint-lem-obstruction} applies at that formula.
\end{proof}

\begin{figure}[H]
	\centering
	\begin{adjustbox}{max width=\linewidth}
	\begin{tikzpicture}[
		node distance=6mm and 12mm,
		every node/.style={font=\small},
		box/.style={draw, rounded corners, inner sep=4pt, align=center},
		hyp/.style={box},
		proofstep/.style={box},
		arr/.style={-{Latex[length=2mm]}, semithick}
		]

		\node[hyp] (evalbot) {$\mathsf{Frame}_{\Bott}(C)$};
		\node[proofstep, below=of evalbot] (existsnegfp)
		{$\exists B\,\NegFP{B}$};
		\node[hyp, below=of existsnegfp] (negfp) {$\NegFP{B}$};
		\node[hyp, left=of negfp] (locallem) {$\LocalLEM{B}$};
		\node[hyp, right=of negfp] (mp) {$\MP{}$};
		\node[proofstep, below=of negfp] (cbot) {$C(\Bott)$};
		\node[proofstep, below=of cbot] (bot) {$\bot$};

		\node[hyp, above=of locallem] (lem) {$\LEM{}$};
		\node[hyp, below=of locallem] (signed) {$\SClass{}$};
		\node[hyp, below=of signed] (membership)
		{$\MemDec{}+\RefComp{}$};
		\node[hyp, below=of membership] (refutation)
		{$\RefInd{}$ (inhabited)};

		\node[hyp, right=of cbot] (cons) {$\Cons{}$};

		\draw[arr] (evalbot) -- (existsnegfp);
		\draw[arr] (existsnegfp) -- (negfp);
		\draw[arr] (negfp) -- (cbot);
		\draw[arr] (locallem) -- (cbot);
		\draw[arr] (mp) -- (cbot);
		\draw[arr] (cbot) -- (bot);
		\draw[arr] (cons) -- (bot);
		\draw[arr] (lem) -- (locallem);
		\draw[arr] (signed) -- (locallem);
		\draw[arr] (membership) -- (signed);
	\end{tikzpicture}
	\end{adjustbox}
\caption{Dependency structure of the local-branch obstruction and its classification consequences.}
\label{fig:obstruction-dependency}
\end{figure}
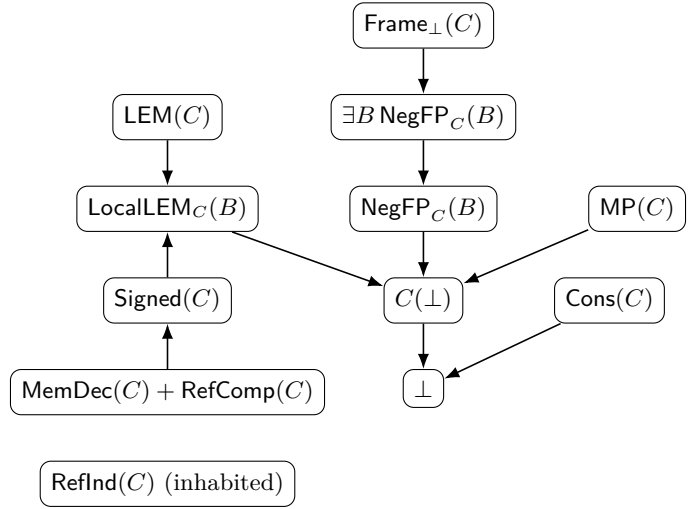

\subsection{Further classification consequences}
\label{sec:l001:corollaries}

\begin{theorem}[Local signed-classification obstruction]
\label{thm:signed-obstruction}
	For any formula $B$, the following hypotheses are incompatible:
	\begin{equation}
		\label{eq:signed-obstruction}
		\NegFP{B} \;\land\; \MP \;\land\; \Cons \;\land\; \SClass \;\implies\; \bot.
	\end{equation}
\end{theorem}

\begin{proof}
	Let $s$ be the signed classifier and inspect $s(B)$. A true verdict supplies $C(B)$; a false verdict supplies $C(\Neg B)$. Thus \LocalLEM{B} holds, and Theorem~\ref{thm:core-diagonal-obstruction} applies.
\end{proof}

\begin{remark}
	The classifier is global as a function, but the proof reads it only once, at the fixed-point $B$. No uniform bivalence schema is required.
\end{remark}

\begin{theorem}[Local refutation-completeness obstruction]
\label{thm:refcomp-obstruction}
	Assume \MP{}, \Cons{}, and $\NegFP{B}$. Then
	\begin{equation}
		\neg\RefComp{}.
	\end{equation}
\end{theorem}

\begin{proof}
	Theorem~\ref{thm:fixedpoint-gap} gives both $\neg C(B)$ and $\neg C(\Neg B)$. If \RefComp{} held, its instance at $B$ would turn $\neg C(B)$ into $C(\Neg B)$, contradicting the second conclusion. No membership decision is needed.
\end{proof}

\begin{lemma}[Membership decision plus refutation completeness gives signed classification]
\label{lem:membership-to-signed}
	If \MemDec{} and \RefComp{} hold, then \SClass{} holds.
\end{lemma}

\begin{proof}
	Let $m$ decide membership. Use the same Boolean function as the signed classifier. A true verdict yields $C(A)$ by the membership specification. If $m(A)=\mathsf{ff}$, then $C(A)$ would force $m(A)=\mathsf{tt}$, so $\neg C(A)$; applying \RefComp{} yields $C(\Neg A)$.
\end{proof}

\begin{corollary}[Membership decision with refutation completeness is obstructed]
\label{thm:membership-obstruction}
	For any formula $B$, the following hypotheses are incompatible:
	\begin{equation}
		\label{eq:membership-obstruction}
	\begin{array}{c}
			\NegFP{B} \;\land\; \MP \;\land\; \Cons \;\land\; \RefComp \;\land\; \MemDec \\[0.75em]
			\Downarrow \\[0.75em]
			 \bot.
	\end{array}
	\end{equation}
\end{corollary}

\begin{proof}
	Theorem~\ref{thm:refcomp-obstruction} already rules out \RefComp{} under the first three hypotheses. The additional membership decider records the classifier comparison but is not needed for the contradiction. Equivalently, Lemma~\ref{lem:membership-to-signed} and Theorem~\ref{thm:signed-obstruction} give the same corollary.
\end{proof}

\begin{theorem}[Refutation-sound Boolean indicators are trivially inhabited]
\label{thm:ref-not-obstructed}
	For every $C : \Formula \to \mathsf{Prop}$, $\RefInd[C]$ is inhabited. Namely, there exists
	\begin{equation}
		r : \Formula \to \{\mathsf{tt},\mathsf{ff}\} \quad\text{such that}\quad \forall A : \Formula,\; r(A) = \mathsf{ff},
	\end{equation}
	and the branch soundness condition $r(A) = \mathsf{tt} \to C(\Neg A)$ holds vacuously.
\end{theorem}

\begin{proof}
	Set $r(A) := \mathsf{ff}$ for every $A$. The soundness condition has an impossible antecedent, so it is satisfied without invoking consistency or any fixed-point.
\end{proof}

\subsection{Sharpness of the local boundary}
\label{sec:l001:sharpness}

The four hypotheses of the core obstruction are individually necessary at the level of arbitrary closure predicates. This can be seen inside the same closed formula language, without adding atoms or semantics.

\begin{proposition}[Sharpness]
\label{prop:sharpness}
	For each of \MP{}, \Cons{}, $\NegFP{B}$, and \LocalLEM{B}, there is a closure predicate for which the other three conditions hold while the selected condition fails.
\end{proposition}

\begin{proof}
	Use $B:=\Bott$ and abbreviate
	\begin{equation}
		T \Defeq \Bott\Imp\Bott,
		\quad U \Defeq \Bott\Imp T,
		\quad V \Defeq T\Imp\Bott.
	\end{equation}
	Thus $T=\Neg B$, while $U=B\Imp\Neg B$ and $V=\Neg B\Imp B$. The four formulas $\Bott,T,U,V$ are syntactically distinct. (i) If consistency is omitted, let $C$ accept every formula. Then \MP{}, $\NegFP{B}$, and \LocalLEM{B} hold, while $C(\Bott)$ makes \Cons{} fail. (ii) If local \textsc{Excluded Middle} is omitted, let $C$ accept exactly $U$ and $V$. This supplies $\NegFP{B}$ and remains consistent. \textsc{Modus Ponens} holds vacuously because the antecedents of the two accepted implications are respectively $\Bott$ and $T$, neither of which is accepted. Hence neither side of \LocalLEM{B} holds. Membership in the finite set $\{U,V\}$ is Boolean-decidable, so this example also shows that \MemDec{} alone is compatible with the other three hypotheses. (iii) If \textsc{Modus Ponens} is omitted, let $C$ accept exactly $T,U,V$. Then \Cons{}, $\NegFP{B}$, and \LocalLEM{B} hold. But $C(V)$ and $C(T)$ hold while $C(\Bott)$ does not, so detachment along $V=T\Imp\Bott$ fails. (iv) If the negation fixed-point is omitted, let $C$ accept exactly $T$. Then \Cons{} and \LocalLEM{B} hold, and \textsc{Modus Ponens} is again vacuous because $T=\Bott\Imp\Bott$ is accepted while its antecedent $\Bott$ is not. Neither $U$ nor $V$ is accepted, so $\NegFP{B}$ fails.
\end{proof}

\section{Mechanization}

\label{sec:mech}

\begin{mechanization}
\label{mech:mech}

The mechanization separates the executable syntactic substrate from the abstract closure obstruction.
\begin{figure}[H]
	\centering
	\vspace{1em}
	\includegraphics[width=3cm]{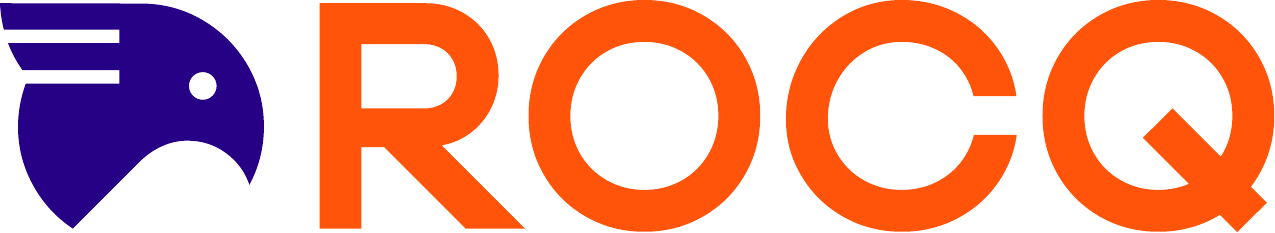}
	\vspace{1em}
	\caption{The \texttt{M001} substrate and the six \texttt{L001} contracts underlying the obstruction are formalized in \textsc{Gallina}, the specification language of the \textsc{Rocq} proof environment.}
	\label{fig:rocq}
\end{figure}
\texttt{M001} provides the closed $\Bott/\Imp$ syntax, the theorem that no formula equals its own negation, a total Boolean proof checker, executable MP and deduction transformations, an executable refutation-to-negation transformation named \texttt{reductio} in the artifact (constructive negation introduction, not a classical principle), syntactic adequacy of checked derivability, context/theory monotonicity, a generic \texttt{SymbolicRegulator} interface, and a pure accepted-output release gate. The theorem \texttt{regulator\_theory\_symbolic\_regulator\_release\_iff} identifies checked derivability with the existence of a finite proof instruction that releases the proposed formula. Its five-contract aggregate \texttt{M001\_CONTRACT} remains \texttt{SYNTAX}, \texttt{ADEQUACY}, \texttt{DEDUCTION}, \texttt{REDUCTIO}, and \texttt{STRUCTURE}, certified by \texttt{m001\_contract\_qed}; release is a derived executable view of Boolean acceptance, not a new logical rule. Its \texttt{AxiomSet} is a caller-supplied Boolean membership function; \texttt{FiniteAxiomSet} is the separate finite-data interface. Extraction remains restricted to the five proof-machine roots in \texttt{M001\_checker.ml/.mli}; neither the regulator packaging nor release is added as an extraction root:
\begin{center}
	\href{https://github.com/Milan-Rosko/Proofcase/tree/main/theories/M001}{\texttt{GitHub $\to$ Milan-Rosko/Proofcase $\to$ M001}}
\end{center}
At the derivability level, \texttt{regulator\_theory\_checked\_derivable\_mp\_lemma} derives Proposition~\ref{prop:m001-checked-mp} from checked proof-script composition. The standard \textsc{K}/\textsc{S} theorems and that composition also yield Proposition~\ref{prop:m001-checked-negfp-collapse}: unlike a bare closure predicate, a checked predicate accepts enough structural implication principles for a negation fixed-point itself to yield $C_{R,\Gamma}(\Bott)$. Corollary~\ref{cor:checked-regulator-obstruction} records this stronger checked boundary.
\texttt{L001} uses only the primitive formula syntax from \texttt{M001}. It defines the closure vocabulary and a goal-frame premise abstractly over $C : \Formula\to\mathsf{Prop}$, then certifies the six contracts listed in Table~\ref{tab:rocq-endpoints}. Its frame contract records self-specialization of the supplied frame, while the derived endpoints \texttt{branchwise\_goal\_relative\_collapse}, \texttt{goal\_relative\_branch\_collapse}, and \texttt{fixedpoint\_gap}, together with the local-detachment bridge, formalize the sharper dependency boundary:
\begin{center}
	\href{https://github.com/Milan-Rosko/Proofcase/tree/main/theories/L001}{\texttt{GitHub $\to$ Milan-Rosko/Proofcase $\to$ L001}}
\end{center}
\begin{table}[H]
\centering
\renewcommand{\arraystretch}{1.4}
\small
\begin{adjustbox}{width=\textwidth}
\begin{tabular}{@{}p{5cm}p{2cm}p{5.2cm}@{}}
	\toprule
	Theorem & Reference & Endpoint \\
	\midrule
	No formula is its own negation & \ref{lem:no-self-neg} & \texttt{formula\_not\_self\_negation} \\
	Local branch collapse & \ref{thm:local-branch-collapse} & \texttt{local\_branch\_collapse} \\
	\textsc{Aporetic Lemma} & \ref{thm:core-diagonal-obstruction} & \texttt{core\_diagonal\_obstruction} \\
	Goal-frame extraction & \ref{thm:eval-bottom-negfixp} & \texttt{eval\_bottom\_negfixp} \\
	Local signed obstruction & \ref{thm:signed-obstruction} & \texttt{local\_signed\_obstruction} \\
	Local membership obstruction & \ref{thm:membership-obstruction} & \texttt{local\_membership\_obstruction} \\
	Refutation-indicator inhabitance & \ref{thm:ref-not-obstructed} & \texttt{closure\_refutation\_inhabited} \\
	\texttt{L001} aggregate contract & \textemdash & \texttt{l001\_contract\_qed} \\
	\bottomrule
\end{tabular}
\end{adjustbox}
\caption{Mapping between the retained contract statements and the reduced \textsc{Rocq} endpoints. The first row belongs to \texttt{M001}; the six closure results and the aggregate belong to \texttt{L001}.}
\label{tab:rocq-endpoints}
\end{table}

The goal-relative, branchwise, local-detachment, and fixed-point-gap results are derived \texttt{L001} theorems rather than additional contract clauses. The paper-level LEM and refutation-completeness corollaries compose these certified results. The stronger checked collapse additionally uses the \texttt{M001} \textsc{K}/\textsc{S} basis and derivability bridge. None enlarges the six-contract \texttt{L001} boundary. Proposition~\ref{prop:sharpness} is likewise a paper-level proof rather than an \texttt{L001} contract. The \texttt{M001} and \texttt{L001} aggregate theorems are \texttt{m001\_contract\_qed} and \texttt{l001\_contract\_qed}. The artifact build is configured to generate assumption reports for the aggregates, and the source contains neither user axioms nor admitted proofs. \texttt{L001} performs no \textsc{OCaml} extraction because its obstruction contracts live in \texttt{Prop}. Its relevant artifact is therefore the assumption report, while \texttt{M001} separately supplies the executable checker artifact.
\end{mechanization}

\section{Discussion}
\label{sec:discussion}

\subsection{Comparisons and outlook}
\label{sec:discussion:comparisons}

Substructural treatments sharpen the comparison. Roberts' refinement of \textsc{Lawvere's Fixed-Point Theorem} \citep{roberts2021} shows that diagonal arguments can survive substantial weakening of surrounding structure. Here the branchwise collapse names only three local detachment instances; global modus-ponens closure is a convenient sufficient condition for them. The proof therefore remains entirely at the closure-detachment level.

Universal-schema presentations of \textsc{Diagonalization}, such as \citet{yanofsky2003}, derive fixed-points from suitable coding and representability conditions. The present theorem begins with the corresponding closure-equivalence point supplied:

\begin{equation}
	B \simeq_C \Neg B.
\end{equation}

The goal frame expresses this premise by self-specialization. Deriving such a frame from a concrete coding apparatus would require additional adequacy or representability assumptions, of the kinds used by the \textsc{Second Incompleteness Theorem} \citep{godel31} and \textsc{Reflection Principles} \citep{feferman62}.

A recursion-theoretic reading would additionally require an effective coding and computation model; no such model is assumed here.

\subsection{Summary and Interpretation}

Indeed, a consistent, modus-ponens-closed primitive regulator admitting $B\simeq_C\Neg B$ cannot validate closure-level \textsc{Excluded Middle}: either verdict at $B$ yields $C(\Bott)$ through the matching fixed-point direction and two detachments. Hence signed classification and refutation completeness are impossible, while ordinary membership decision alone remains compatible. Refutation-sound Boolean indicators are inhabited by the always-false function.

This clarifies the intuitionistic restraint. The constructive refusal is not a competing answer to \LEM{}; it is the refusal to assert the schema uniformly. We may conclude that a consistent $C$ must withhold disjuncts:

\begin{quotation}
	A regulator must not accept either branch at a formula if it is to remain consistent.
\end{quotation}

That diagnosis is not a $C$-fact unless it is imported into $C$. The obstruction begins precisely when a branch verdict is required internally and uniformly. The collapse is therefore not a consequence of “too much” semantics, but of insufficient discrimination: the regulator treats a certain formula as imposing a merely formal closure condition rather than as a \emph{philosphy}.

\section*{References and Notes}

	{\scriptsize
		\bibliographystyle{plainnat}
		\setlength{\bibsep}{0.1em}
		\bibliography{refs}}

\clearpage

\vspace*{\fill}

\subsection*{Note on the Third Version}

This third version brings the exposition into line with the streamlined \texttt{M001} and \texttt{L001} developments. The collapse now works for any target formula, each case isolates the implication and detachment steps it uses, and consistency rules out both sides at the fixed-point. The goal-frame and refutation-sound endpoints now match their streamlined contracts. The exposition also separates the abstract closure boundary from the stronger checked \textsc{K}/\textsc{S} consequence. Derived contract theorems support these paper-level consequences, with the five \texttt{M001} contracts and six \texttt{L001} contracts unchanged.

\subsection*{Final Remarks}

The author welcomes criticism, proposed extensions, scholarly correspondence, and constructive dialogue. No conflicts of interest are declared. This research received no funding.

\begin{center}
	\scriptsize{
	\vspace{1em}
	Milan Rosko\\
	\vspace{1em}
	ORCID: \href{https://orcid.org/0009-0003-1363-7158}{\textsf{0009-0003-1363-7158}}\\[1ex]
	Email: \href{mailto:hi*atsymbol*milanrosko.com}{\textsf{hi*atsymbol*milanrosko.com}}\\[1ex]

	\vspace{1em}
	Licensed under \ccby \\ \href{http://creativecommons.org/licenses/by/4.0/}{\scriptsize\textsf{creativecommons.org/licenses/by/4.0}}
	}
\end{center}
\vspace*{\fill}

\end{document}